\documentclass[12pt,leqno]{article}
\usepackage{amssymb,amsmath,amsthm}
\usepackage[all]{xy}
\usepackage[alphabetic,lite]{amsrefs}
\numberwithin{equation}{section}
\usepackage[top=2cm,bottom=2cm,left=2cm,right=4cm,marginparsep=0.3cm,marginparwidth=3cm,includefoot]{geometry}
\usepackage{mathtools}
\usepackage{dq2}

\usepackage[colorlinks=true, pdfstartview=FitV, linkcolor=blue, citecolor=blue, urlcolor=blue]{hyperref}
\usepackage[normalem]{ulem} \usepackage{mathrsfs}
\usepackage{accents}
\usepackage{xspace}
\usepackage{marginnote}

\date{March 6, 2020}
\begin{document}

\title{A finiteness theorem for holonomic DQ-modules on Poisson manifolds}

\author{Masaki Kashiwara and Pierre Schapira%
}

\maketitle

\begin{abstract}
On a complex symplectic manifold, 
we prove a finiteness result for the global sections of solutions of holonomic \DQ-modules in two cases: (a) by assuming that there exists a Poisson compactification  (b) in the algebraic case. This extends our previous result of~\cite{KS12} in which the symplectic manifold  was compact. 
The main tool is  a finiteness theorem for $\R$-constructible sheaves on a real analytic manifold in a non proper situation. 
\end{abstract}

{\renewcommand{\thefootnote}{\mbox{}}
\footnote{The research of M.K 
was supported by Grant-in-Aid for Scientific Research (X) 15H03608, Japan Society for the Promotion of Science.}
\footnote{The research of P.S was supported by the  ANR-15-CE40-0007 ``MICROLOCAL''.}
\addtocounter{footnote}{-3}
}

\tableofcontents

\section{Introduction and statement of the results}
Consider a complex Poisson manifold $X$   of complex dimension $d_X$  endowed with a DQ-algebroid $\A[X]$.
Recall that $\A[X]$ is a $\C\forl$-algebroid 
locally isomorphic  to a star algebra $(\sho_X\forl,\star)$ to which the Poisson structure is associated. Denote by $\Ah[X]$ the localization of $\A[X]$ with respect to  $\hbar$, a $\C\Ls$-algebroid. For short, we set
\eqn
&&\Chbar\eqdot\C\forl,\quad \Chbarl\eqdot\C\Ls.
\eneqn
Hence  $\Ah[X]\simeq\Chbarl\tens[\Chbar]\A$. 
The  algebroids $\A[X]$ and $\Ah[X]$ are right and left Noetherian (in particular coherent) and if $\shm$ is a (say left) coherent $\Ah[X]$-module, then its support is a closed complex analytic subvariety of $X$ and it follows from Gabber's theorem that it is co-isotropic. In the extreme case where $X$ is symplectic and the support is Lagrangian, one says that $\shm$ is holonomic. 

Recall the following definitions (see~\cite{KS12}*{Def.~2.3.14,~2.3.16 and~2.7.2}).
\banum
\item
A coherent $\A[X]$-submodule $\shm_0$ of a coherent $\Ah[X]$-module $\shm$ is called
an $\A[X]$-lattice of~$\shm$ if $\shm_0$ generates $\shm$.
\item
A coherent $\Ah[X]$-module $\shm$ is good
if, for any relatively compact open subset $U$ of~$X$, there exists
an $(\A[X]\vert_U)$-lattice of~$\shm\vert_U$.
\item
One denotes by $\RD^\Rb_{\gd}(\Ah[X])$ the full subcategory of~$\RD^\Rb_{\coh}(\Ah[X])$ consisting of objects with  good  cohomology. 
\item
In the algebraic case (see below) a coherent $\Ah$-module $\shm$ is called algebraically good
if there exists an $\A[X]$-lattice of~$\shm$. One still denotes by $\RD^\Rb_{\gd}(\Ah[X])$ the full subcategory of~$\RD^\Rb_{\coh}(\Ah[X])$ consisting of objects with  algebraically good  cohomology. 
\eanum
 Let $Y\subset X$. We shall consider the hypothesis
\eq\label{hyp:Ysalc}&&
\parbox{75ex}{
$Y$ is open, relatively compact, subanalytic in $X$ and the Poisson structure on $X$ is symplectic on $Y$.
}\eneq
\begin{example}
Denote  by $\Xns$ the closed complex subvariety of $X$  consisting of points where the Poisson bracket is not symplectic and set $Y=X\setminus\Xns$. 
Hence $Y$ is an open subanalytic subset of $X$ and is symplectic. 
If $X$ is compact, then $Y$ satisfies hypothesis~\eqref{hyp:Ysalc}. 
\end{example}
In this paper we shall prove the following theorem which extends~\cite{KS12}*{Th.~7.2.3} in which $X$ was symplectic, that is, $Y=X$.

\begin{theorem}\label{th:Main1}
Assume  that  $Y$ satisfies  hypothesis~\eqref{hyp:Ysalc}.
Let~$\shm$ and $\shl$ belong to $\Derb_\gd(\Ah)$ and assume that both   $\shm\vert_Y$ and $\shl\vert_Y$  are holonomic. Then the two complexes $\RHom[{\Ah[Y]}](\shm\vert_Y,\shl\vert_Y)$ and $\rsect_c(Y;\rhom[{\Ah[Y]}](\shl\vert_Y,\shm\vert_Y))\,[d_X]$ 
have finite dimensional cohomology over $ \Chbarl$ and are dual to each other. 
\end{theorem}

We shall also obtain a similar conclusion under  rather different hypotheses, namely that $X=Y$ is symplectic  and all data are algebraic  (see~\cite{KS12}*{\S~2.7}).
 Let $X$ be a smooth algebraic variety and  let 
$\A$ be a   DQ-algebroid on $X$.
We denote by $\Xan$ the associated complex analytic manifold and $\A[\Xan]$ the associated DQ-algebroid on $\Xan$ (see Lemma~\ref{lem:analy}). 
For a coherent 
$\Ah$-module $\shm$ we denote by $\shm_\an$ its image by the natural functor $\Derb_\coh(\Ah)\to\Derb_\coh(\Ah[\Xan])$. 

\begin{theorem}\label{th:Main2}.
Let $X$ be a quasi-compact separated smooth symplectic algebraic variety over $\C$ endowed with the Zariski topology. 
Let~$\shm$ and $\shl$ belong to $\Derb_\gd(\Ah)$. Then the two complexes  $\RHom[{\Ah[\Xan]}](\shm_\an,\shl_\an)$ 
 and $\rsect_c(\Xan;\rhom[{\Ah[\Xan]}](\shl_\an,\shm_\an))\,[d_X]$ 
have finite dimensional cohomology over $ \Chbarl$ and are dual to each other. 
\end{theorem}

The main tool in the proof of both theorems is Theorem~\ref{th:sa} below which gives a finiteness criterion for $\R$-constructible sheaves on a real analytic manifold in a non proper situation. 

This Note is motivated by the  paper~\cite{GJS19} of Sam Gunningham,  David Jordan and Pavel Safronov on Skein algebras, whose main theorem is based over such a finiteness result (see loc.\ cit.~\S~3). The proof of these authors uses a kind of Nakayama theorem in the case where $\shm$ and $\shl$ are simple modules over smooth Lagrangian varieties. 

\section{Finiteness results for constructible sheaves}
In this paper,   $\cor$  is a Noetherian
commutative ring of finite global  homological dimension.

We denote by 
$\Derb_f(\cor)$ the full triangulated subcategory of $\Derb(\cor)$ consisting of objects with finitely generated cohomology.  We denote by  $\RD$  the duality functor $\RHom(\scbul,\cor)$ and we say that two objects $A$ and $B$ of $\Derb_f(\cor)$ are dual to each other if $\RD A\simeq B$, which is equivalent to $\RD B\simeq A$.

For a sheaf of rings $\shr$, one
denotes by $\RD(\shr)$ the derived category of left $\shr$-modules. We shall also encounter the full triangulated subcategory $\RD^+(\shr)$ or $\Derb(\shr)$ of complexes whose cohomology is bounded from below or is bounded.

For  a real analytic manifold $M$, one denotes by $\Derb(\cor_M)$ the bounded derived category of sheaves of $\cor$-modules on $M$. 
We shall use the six Grothendieck operations. In particular, we denote by $\omega_M$ the dualizing complex. We also use the notations
for $F\in \Derb(\cor_M)$ 
\eqn
&&\RD_M'F\eqdot\rhom(F,\cor_M),\quad \RD_M F\eqdot\rhom(F,\omega_M).
\eneqn
Recall that an object $F$ of~$\Derb(\cor_M)$ is weakly $\R$-constructible if condition (i)  below is satisfied. If moreover condition (ii) is satisfied, then one says that $F$ is $\R$-constructible.
\bnum
\item
 there exists a subanalytic stratification
$M=\bigsqcup_{a\in A}M_a$ such that
$H^j(F)\vert_{M_a}$ is  locally constant for all $j\in\Z$ and all
$a\in A$
\item
 $H^j(F)_x$ is finitely generated for all
$x\in M$ and all $j\in\Z$.
\enum
One denotes by~$\Derb_\Rc(\cor_M)$
the full subcategory of~$\Derb(\cor_M)$ consisting of~$\R$-constructible objects.  

If $X$ is a complex analytic manifold, one defines similarly the notions of (weakly) 
 $\C$-constructible sheaf, replacing ``subanalytic''
with ``complex analytic'' and one  denotes by~$\Derb_\Cc(\cor_X)$
the full subcategory of~$\Derb(\cor_X)$
consisting of $\C$-constructible  objects. 

 We shall use the following classical result (see~\cite{KS90}*{Prop.~8.4.8 and Exe.~VIII.3}).
 
\begin{proposition}\label{pro:dualrc}
Let $F\in\Derb_\Rc(\cor_M)$ and assume that $F$ has compact support.  Then both 
objects $\rsect(M;F)$ and $\rsect(M;\RD_M F)$ belong to $\Derb_f(\cor)$ and are dual to each other. 
\end{proposition}

For $F\in\Derb(\cor_M)$, one denotes by $\SSi(F)$ its microsupport \cite{KS90}*{Def.~5.1.2},  a
closed $\R^+$-conic ({\em i.e.,} invariant by the $\R^+$-action on~$T^*M$)
subset of~$T^*M$. Recall that this set
is involutive (one also says {\em co-isotropic}), 
see~ loc.\ cit.\ Def.~6.5.1.

\begin{theorem}\label{th:sa}
Let $j\cl U\into M$ be the embedding  of  an open subanalytic subset $U$ of $M$ and let $F\in\Derb_\Rc(\cor_U)$. Assume that 
$\SSi(F)$ is contained in a closed subanalytic  $\R^+$-conic  Lagrangian subset $\Lambda$ of $T^*U$ which is subanalytic in $T^*M$. 
Then $\roim{j}F$ and $\eim{j}F$ belong to $\Derb_\Rc(\cor_M)$. 
\end{theorem}

\begin{proof}
(i) Let us treat first $\eim{j}F$. The set $\Lambda$ is a locally closed subanalytic subset of $T^*M$ and is isotropic. By~\cite{KS90}*{Cor.~8.3.22}, there exists a $\mu$-stratification $M=\bigsqcup_{a\in A}M_a$ such that 
$\Lambda\subset\bigsqcup_{a\in A}T^*_{M_a}M$.

Set $U_a=U\cap M_a$. Then $U=\bigsqcup_{a\in A}U_a$ is a  $\mu$-stratification  and one can 
apply loc.\ cit.\ Prop.~8.4.1. Hence, for each $a\in A$, $F\vert_{U_a}$ is locally constant of finite rank. 
Hence  $(\eim{j}F)\vert_{U_a}$ as well as $(\eim{j}F)_{M\setminus U}\simeq0$ is
locally constant of finite rank. 
Hence $\eim{j}F\in \Derb_\Rc(\cor_M)$. 

\spa
(ii) Set $G=\eim{j}F$. Then $G\in\Derb_\Rc(\cor_M)$ by (i) and so does $\roim{j}F\simeq \rhom(\cor_U,G)$ (apply~\cite{KS90}*{Prop.~8.4.10}).
\end{proof}
\begin{remark}
One has $\SSi(\RD_M F)=\SSi(F)^a$ where $(\scbul)^a$ is the antipodal map. Hence $\RD_M F$ satisfies the same hypotheses as $F$.
\end{remark}

\begin{corollary}\label{cor:finite1}
In the preceding situation, assume moreover that 
$U$ is relatively compact in $M$. Then $\rsect(U;F)$ and $\rsect_c(U;\RD_U F)$ belong to $\Derb_f(\cor)$ and are dual to each other.
\end{corollary}
\begin{proof}
One has $\rsect(U;F)\simeq \rsect(M;\roim{j}F)$ and  $\rsect_c(U;\RD_UF)\simeq \rsect(M;\RD_M\roim{j}F)$. Since $\roim{j}F$ 
is $\R$-constructible and has compact support, the result follows from Proposition~\ref{pro:dualrc}.
\end{proof}

For a complex analytic manifold $X$ (that we identify with the real underlying manifold if necessary), one denotes by  
$\Derb_\Cc(\cor_X)$  the full triangulated subcategory of $\Derb(\cor_X)$  consisting of $\C$-constructible sheaves. 

In this paper, a smooth  algebraic variety $X$ means  a quasi-compact smooth algebraic variety over $\C$ endowed with the Zariski topology. We denote by $\Xan$ the complex analytic manifold underlying $X$. 
If $X$ is  smooth  algebraic  variety, we keep the notation 
$\Derb_\Cc(\cor_X)$ to denote the category of algebraically constructible sheaves, that is, object of $\Derb_\Cc(\cor_\Xan)$ locally constant on an algebraic stratifications. Hence, for an algebraic variety $X$, one shall not confuse $\Derb_\Cc(\cor_X)$ and 
$\Derb_\Cc(\cor_\Xan)$, although $\Derb_\Cc(\cor_X)$ is a full subcategory of
$\Derb_\Cc(\cor_\Xan)$.

\begin{corollary}\label{cor:finite2}
 Let $X$ be a smooth  algebraic variety and let  $F\in\Derb_\Cc(\cor_X)$. Then 
 $\rsect(\Xan;F)$  and  $\rsect_c(\Xan;\RD_\Xan F)$ have finite dimensional cohomology over $\cor$  and are dual to each other.
\end{corollary}
\begin{proof}
Let $Z$ be a smooth algebraic compactification of $X$ with $X$ open in $Z$. 
By the hypothesis, $\Lambda$ is a closed algebraic subvariety of $T^*X$. Hence, 
 its  closure in $T^*Z$ is a closed  algebraic subvariety of $T^*Z$. Therefore $\Lambda$ is subanalytic in $T^*\Zan$.
 
Then the result follows from  Corollary~\ref{cor:finite1} with $M=\Zan$ and $U=\Xan$.
\end{proof}

\section{Reminders on DQ-modules, after~\cite{KS12}}

\subsection{Cohomologically complete modules}

In this subsection,
\eq
&&\parbox{75ex}{$X$ denotes  a  topological space and 
$\shr$ is a sheaf of $\Z[\hbar]$-algebras on $X$ with no $\hbar$-torsion.
}
\eneq

Let $\shm$ be an $\shr$-module. (Hence, a $\Z_X[\hbar]$-module.) One sets
\eqn
&&\Rh\eqdot \Z_X[\hbar,\hbar^{-1}]\tens[{\Z_X[\hbar]}]\shr,\\
&&\shm^\loc\seteq  \Rh\tens[\shr]\shm\simeq   \Z_X[\hbar,\hbar^{-1}]\tens[{\Z_X[\hbar]}]\shm,\\
&&\gr(\shr)\eqdot  \shr/\hbar\shr,\\
&&\gr(\shm)\seteq \gr(\shr)\ltens[\shr]\shm\simeq\Z_X\ltens[{\Z_X[\hbar]}]\shm.
\eneqn

\begin{definition}[{\cite{KS12}*{Def.~1.5.5}}]\label{def:cohco}
One says that  an object $\shm$ of~$\Der(\shr)$
is cohomologically complete
if it belongs to~$\Der(\Rh)^{\perp r}$, that is, 
$\Hom[\Der(\shr)](\shn,\shm)\simeq0$ for any $\shn\in\Der(\Rh)$.
\end{definition}

\begin{proposition}[{\cite{KS12}*{Prop.~1.5.6}}]\label{pro:eqvcoc}
Let $\shm\in\Der(\shr)$. 
Then the conditions below are equivalent.
\banum
\item
 $\shm$ is  cohomologically complete,
\item
 $\rhom[\shr](\Rh,\shm)\simeq0$,
\item
$\indlim[U\ni x]\Ext[{\Z[\hbar]}]{j}\bl{\Z[\hbar,\opb{\hbar}]},H^i(U;\shm)\br\simeq0$
for any $x\in X$, $j=0,1$ and any $i\in\Z$. Here, $U$ ranges over an open
neighborhood system of~$x$.
\eanum
\end{proposition}

Denote by $\Der_\coc(\shr)$ the full subcategory of  $\Der(\shr)$ consisting of  cohomologically complete modules. Then clearly 
$\Der_\coc(\shr)$ is triangulated. 

\begin{proposition}[{\cite{KS12}*{Prop.~1.5.10, Cor.~1.5.9}}]\label{prop:homcc}
Let  $\shm\in\Der_\coc(\shr)$. Then
\banum
\item
$\rhom[\shr](\shn,\shm)\in\Der(\Z_X[\hbar])$ is \cc\
for any $\shn\in\Der(\shr)$.
\item
If $\gr(\shm)\simeq 0$, then $\shm\simeq0$.
\eanum
\end{proposition}

\begin{proposition}[{\cite{KS12}*{Prop.~1.5.12}}]\label{pro:cohcodirim}
Let~$f\cl X\to Y$ be a continuous map and let $\shm\in\Der(\Z_X[\hbar])$.
If $\shm$ is \cc, then so is $\roim{f}\shm$.
\end{proposition}

\begin{proposition}\label{pro:158}
Let $\shm\in\Der(\shr)$ be a cohomologically complete object and $a\in\Z$.
If $H^i(\gr(\shm))=0$ for any $i\geq a$, then $H^i(\shm)=0$ for any
$i>a$.
\end{proposition}
\Proof
The proof is exactly the same as that of~\cite{KS12}*{Prop.~1.5.8}
when replacing $i>a$ with $i<a$.
\QED

\subsection{Microsupport and constructible sheaves}

Let~$M$ be a {\em real analytic} manifold and let  $\cor$ be a Noetherian
commutative ring of finite global homological dimension.

We shall need the next result which does not appear in~\cite{KS12}.
\begin{proposition}\label{pro:ssloc}
Let~$F\in\Derb(\Z_M[\hbar])$.
Then $ \SSi(F^\loc)\subset \SSi(F)$.
\end{proposition}
\begin{proof}
By using one of the equivalent definitions of the micro-support given in~\cite{KS90}*{Prop.~5.11}, it is enough to check that for $K$ compact, 
$\rsect(K;F)^\loc\simeq\rsect(K;F^\loc)$ which follows from loc.\ cit.\ Prop.~2.6.6 and the fact that 
$\Z[\hbar,\opb{\hbar}]$ is flat over $\Z[\hbar]$. 
\end{proof}

\begin{proposition}[{\cite{KS12}*{Prop.~7.1.6}}]\label{pro:sscohco}
Let~$F\in\Derb(\Z_M[\hbar])$ and assume that $F$ is cohomologically
complete. Then
\eq\label{eq:SSgr}
&&\SSi(F)=\SSi(\gr(F)).
\eneq
\end{proposition}
\begin{proof}
Let us recall the proof of loc.\ cit. The inclusion
\eqn
&&\SSi(\gr(F)) \subset \SSi(F)
\eneqn
follows from the distinguished triangle
$F\to[\hbar]F\to\gr(F)\to[+1]$.
Let us prove the converse inclusion.

Using the definition of the microsupport, it is enough to prove that given two
open subsets $U\subset V$ of~$M$,
$\rsect(V;F)\to\rsect(U;F)$ is an isomorphism
as soon as
$\rsect(V;\gr(F))\to \rsect(U;\gr(F))$ is an isomorphism.
Consider a distinguished triangle $\rsect(V;F)\to\rsect(U;F)\to G\to[+1]$.
Then we get a distinguished triangle
$\rsect(V;\gr(F))\to \rsect(U;\gr(F))\to \gr(G)\to[+1]$. Therefore,
$\gr(G)\simeq0$. On the other hand,
$G$ is cohomologically complete, thanks to
Proposition~\ref{pro:cohcodirim} (applied to  $F\vert_U$ and $F\vert_V$) and  then $G\simeq0$ by Proposition~\ref{prop:homcc}~(b).
\end{proof}

\begin{proposition}[{\cite{KS12}*{Prop.~7.1.7}}]\label{pro:RCcohco}
Let~$F\in \Derb_\Rc(\Chbar)$. Then $F$ is cohomologically complete.
\end{proposition}
\begin{proof}
Let us recall the proof of loc.\ cit.  One has
\eqn
\inddlim[U\ni x]\Ext[{\Z[\hbar]}]{j}\bl\Z[\hbar,\hbar^{-1}],H^i(U;F)\br&\simeq&
\Ext[{\Z[\hbar]}]{j}\bl\Z[\hbar,\hbar^{-1}],\inddlim[U\ni x]H^i(U;F)\br\\
&\simeq&\Ext[{\Z[\hbar]}]{j}\bl\Z[\hbar,\hbar^{-1}],F_x\br
\simeq0
\eneqn
where the last isomorphism follows from the fact that $F_x$ is cohomologically complete.

Hence, hypothesis~(c) of Proposition~\ref{pro:eqvcoc} is satisfied.
\end{proof}

\subsection{DQ-modules}

{}In this subsection,  $X$ will be  a complex   manifold (not necessarily symplectic) of complex dimension $d_X$. 

Set 
$\Oh[X]\eqdot\sho_X[[\hbar]]=\prolim[n]\sho_X\otimes_{\C}(\Chbar/\hbar^n\Chbar)$.
An associative multiplication law 
$\star$ on $\Oh[X]$ is a
star-product if it is $\Chbar$-bilinear and satisfies
\eq\label{eq:star1}
&&f\star g=\sum_{i\geq 0}P_i(f,g)\hbar^i\quad\text{for $f,g\in\OO[X]$,}
\eneq
where the $P_i$'s are bi-differential operators, $P_0(f,g)=fg$ 
and $P_i(f,1)=P_i(1,f)=0$ for $f\in\sho_X$ and $i>0$. 

 We call $(\OO[X][[\hbar]],\star)$ a {\em star-algebra}. 
A $\star$-product defines a Poisson structure on $(X,\sho_X)$ by the formula
\eq\label{eq:poisson}
&&\{f,g\}=P_1(f,g)-P_1(g,f)\equiv\opb{\hbar}(f\star g-g\star f)\mod\hbar\OO[X][[\hbar]].
\eneq

\begin{definition}\label{def:DFDalg}
A $\DQ$-algebroid  $\sha$ on $X$ is a $\Chbar$-algebroid locally
isomorphic to a star-algebra as a $\Chbar$-algebroid.
\end{definition}
\begin{remark}
The data of a \DQ-algebroid $\A$ on $X$ endows $X$ with a structure of a complex Poisson manifold and one says that $\A$ is a quantization of the Poisson manifold.  
Kontsevich's famous theorem~\cite{Ko01, Ko03} (see also~\cite{Ka96} for the case of contact manifolds) asserts that  any complex Poisson manifold may be quantized.
\end{remark}
\begin{example}\label{exa:1}
Assume that $M$ is an open subset of $\C^n$,   $X=T^*M$ and denote by  $(x;u)$
 the symplectic coordinates on $X$.  In this case there is a canonical $\star$-algebra $\A$ that is usually  denoted by 
$\HWo$, its localization with respect to $\hbar$ being denoted by $\HW$. 

Let~$f,g\in\OO[X]\forl$. Then the $\DQ$-algebra $\HWo$ is the star algebra $(\OO[X]\forl,\star)$ where:
\eq\label{eq:wstar}
f\star g
&=&\sum_{\alpha\in\N^n} \dfrac{\hbar^{\vert\alpha\vert}}{\alpha !}
(\partial^{\alpha}_uf)(\partial^{\alpha}_xg).
\eneq
This product is similar to the product of the total symbols of
differential operators on~$X$ and indeed, the
morphism of~$\C$-algebras $\pi_M^{-1}\shd_M\To \HW[X]$ is given by
\eqn
f(x)\mapsto f(x),&&\partial_{x_i}\mapsto \hbar^{-1}u_i,
\eneqn
where, as usual, $\shd_M$ denotes the ring of finite order holomorphic differential operators and $\pi_M\cl T^*M\to M$ is the projection.
\end{example}

\medskip
 For a $\DQ$-algebroid $\A$, there is locally an  isomorphism of  $\C$-algebroids
$\A/\hbar \A\isoto \sho_X$.  Moreover there exists a unique isomorphism of $\C$-algebras
\eq\label{eq:OendA}
&&\shend(\id_{\gr\A})\simeq\sho_X.
\eneq
Therefore, there is  a well-defined functor
\eq\label{eq:OtensA}
&&\scbul\ltens[\sho_X]\scbul\cl\Derb(\sho_X)\times\Derb(\gr\A)\to\Derb(\gr\A).
\eneq

\begin{theorem}[{\cite{KS12}*{Th.~1.2.5}}]\label{th:ks125}
For a $\DQ$-algebroid $\A$, both $\A$ and $\Ah$ are right and left Noetherian 
\lp in particular, coherent\rp. 
\end{theorem}
One defines the functors
\eqn
&&\gr\cl\Derb(\A)\to\Derb(\gr\A),\quad \shm\mapsto \C_X\ltens[\Chbar_X]\shm,\\
&&(\scbul)^\loc\cl \Derb(\A)\to\Derb(\Ah),\quad  \shm\mapsto  \Chbarl_X\tens[\Chbar_X]\shm,\\
&&\for\cl  \Derb(\bA)\to \Derb(\A)\mbox{ associated with }\sigma_0\cl \A[X]\to\bA.
\eneqn
The functor $(\scbul)^\loc$ is exact on  $\md[\A]$.
  The category $\Mod(\bA)$ is equivalent to the
full subcategory of~$\Mod(\A)$ consisting of objects $M$ such that $\hbar\cl M\to M$ vanishes.

\begin{theorem}[{\cite{KS12}*{Th.~1.6.1 and~1.6.4}}]\label{th:formalfini2}
Let~$\shm\in\Der^+(\A)$. Then the two conditions below are equivalent: 
\banum
\item[{\rm(a)}]
$\shm$ is cohomologically complete
and $\gr(\shm)\in\Der^+_\coh(\gr\A)$,
\item[{\rm(b)}]
$\shm\in\Der^+_\coh(\A)$. 
\eanum
\end{theorem}

The next result follows from Gabber's theorem~\cite{Ga81}.

\begin{proposition}[{\cite{KS12}*{Prop.~2.3.18}}]\label{pro:invol}
Let~$\shm\in\RD^\Rb_{\coh}(\Ah[X])$. Then $\supp(\shm)$ \lp the support of $\shm$\rp\, is a closed
complex analytic subset of~$X$, involutive \lp{\em i.e.,} co-isotropic\rp\, for
the Poisson bracket on~$X$.
\end{proposition}
\begin{remark}
One shall be aware that the support of a coherent
$\A$-module is not involutive in general. Indeed,
any coherent $\gr\A$-module may be
regarded as an $\A$-module.
Hence any closed analytic subset can be the support of a coherent $\A$-module.
\end{remark}

\section{\DQ-modules along $\Lambda$}

\subsection{A variation on a theorem of~\cite{Ka03}}

In order to prove Lemma~\ref{le:grAL1} below, we need  a slight modification of a result of~\cite{Ka03}.

Let $\shr$ be a ring on a topological space $X$, and
let $\st{F_n(\shr)}_{n\in\Z}$ be a filtration of $\shr$ which satisfies
\bna
\item $\shr=\bigcup_{n\in\Z}\Fi{n}$,
\item $1\in\Fi{0}$,
\item $\Fi{m}\cdot\Fi{n}\subset\Fi{m+n}$.
\ee
We set 
\eqn
&&\grF_{\ge0}(\shr)=\soplus_{n\ge0}\grF_n(\shr).
\eneqn

\begin{proposition}\label{pro:ka03}
Assume that
\banum
\item $\Fi{0}$ and $\grF_{\ge0}(\shr)$ are Noetherian rings,
\item $\grF_n(\shr)$ is a coherent $\Fi{0}$-module for any $n\ge0$.
\eanum
Then $\shr$ is Noetherian.
\end{proposition}
\begin{proof}
Define $\tw F_n(\shr)$ by
$$\tw F_n(\shr)=\bc
\Fi{n}&\text{if $n\ge0$,}\\
0&\text{if $n<0$.}\\
\ec$$

We shall apply \cite{Ka03}*{Theorem~A.20}
to $\tw F_k(\shr)$.
Hence in order to prove the theorem, it is enough to show
\eq
&&\text{\parbox{73ex}{for any positive integer $m$ and an open subset $U$ of $X$,
if an $\shr\vert_U$-submodule $\shn$ of $\shr^{\oplus m}\vert_U$
has the property that $F_k(\shn)\seteq\shn\cap\Fi{k}^{\oplus m}\vert_U$
is a coherent  $\Fi{0}\vert_U$-module for any $k\ge0$,
then $\shn$ is a locally finitely generated $\shr\vert_U$-module.}}
\label{noeth}
\eneq
Since $\grF_{\ge0}(\shr)$ is a Noetherian ring,
$\grF_{\ge0}(\shn)\seteq\soplus_{n\ge0}\grF_n(\shn)$ is a
coherent $\grF_{\ge0}(\shr)$-module.
Hence there exists locally a finitely generated $\shr$-submodule $\shn'$ of 
$\shn$
such that
$\grF_{\ge0}(\shn')=\grF_{\ge0}(\shn)$.
Hence we have $\shn=\shn'+F_{0}(\shn)$.
Since $F_{0}(\shn)$ is a locally finitely generated $\Fi{0}$-module,
$\shn$ is 
locally finitely generated $\shr$-module.
\end{proof}

\subsection{The algebroid $\AL$}
From now on, $X$ is a complex  manifold endowed with a \DQ-algebroid $\A$.

\begin{definition}[{\cite{KS12}*{Def.~2.3.10}}]\label{def:simplemod}
Let~$\Lambda$ be a smooth submanifold of~$X$  and
let~$\shl$ be a coherent $\A[X]$-module supported by~$\Lambda$.
One says that $\shl$ is simple along $\Lambda$ if
$\gr(\shl)$ is concentrated in degree $0$ and $H^0(\gr(\shl))$
is an invertible $\sho_\Lambda\tens[{\sho_X}]\gr(\A)$-module. (In particular,
$\shl$ has no $\hbar$-torsion.)
\end{definition}

Let~$\Lambda$ be a smooth 
submanifold of~$X$ and let~$\shl$ be  a coherent $\A[X]$-module
simple along $\Lambda$.
We set for short
\eqn
&&\Oh[\Lambda]\eqdot\sho_\Lambda\forl,\quad \Ohl[\Lambda]\eqdot\sho_\Lambda\Ls,\\
&&\DDh[\Lambda]\eqdot\shd_\Lambda\forl,\quad \DDhl[\Lambda]\eqdot\shd_\Lambda\Ls.
\eneqn

One proves that there is a natural isomorphism  of algebroids $\shend_{\CORO}(\shl)\simeq\shend_{\CORO}(\Oh[\Lambda])$ 
(\cite{KS12}*{Lem. 2.1.12}). 
Then the subalgebroid of~$\shend_{\CORO}(\shl)$ corresponding to the
subring
$\Dh[\Lambda]$ of~$\shend_{\CORO}(\Oh[\Lambda])$
is well-defined. We denote it by~$\DL$. Then (see~\cite{KS12}*{Lem.~7.1.1}):
\banum
\item
$\DL$ is isomorphic  to~$\DDh[\Lambda]$ as a $\CORO$-algebroid and $\gr(\DL)\simeq\shd_\Lambda$.
\item
The $\CORO$-algebra  $\DL$ is right and left Noetherian.
\eanum

We denote by~$I_\Lambda\subset \sho_X$ the defining ideal of~$\Lambda$.
Let~$\shi$ be the kernel of the composition
\eq\label{eq:shi}
&& \opb{\hbar}\A[X]\To[\hbar]\A[X]\To\gr\A\to\OO[\Lambda]\ltens[\sho_X]\gr\A.
\eneq
Then we have 
\eq\label{eq:shi2}
&&\shi/\A[X]\simeq I_\Lambda\tens[\sho_X]\gr\A.
\eneq
\begin{remark}
In~\cite{KS12}*{Ch.~7,~\S~1} we have used the symbol map $\sigma\cl\A\to\OO$. This map is only defined locally, but all results of this chapter are  of local nature. If nevertheless, one wants a global construction, then one has to replace the sequence two lines above Definition 7.1.2 of loc.\ cit.\ with~\eqref{eq:shi}.
\end{remark}
\begin{definition}[{\cite{KS12}*{Def.~7.1.2}}]
One denotes by~$\AL$ 
the $\CORO$-subalgebroid of~$\Ah[X]$ generated by~$\shi$.
\end{definition}

The ideal $\hbar\shi$ is contained in~$\A[X]$, hence acts on~$\shl$ and
one sees easily that $\hbar\shi$ sends $\shl$ to~$\hbar\shl$.
Hence, $\shi$ acts on~$\shl$ and defines a  functor $\AL\to\DL$.
We thus have the morphisms of algebroids
\eqn
&&\xymatrix{
{\A[X]\vert_\Lambda}\ar[r]\ar[dr]&{\AL\vert_\Lambda}\ar[d]\ar[r]&{\Ah[X]\vert_\Lambda}\ar[d]\\
&{\DL}\ar[r]&{\DLl}\,.
}\eneqn
In particular, $\shl$ is naturally an $\AL$-module    and $\gr(\DL)\simeq\shd_\Lambda$ is a $\gr(\AL)$-module.

\begin{example}\label{exa:2}
 We follow the notations of  Example~\ref{exa:1}. 
Let $\Lambda=M$. Then $ \shl\seteq\HWo/\bigl(\sum_i\HWo u_i\bigr)
\simeq \Oh[\Lambda]$ is simple along $\Lambda$ 
and  $\shi\subset\opb{\hbar}\A=\A(-1)$ is generated by $\opb{\hbar}u=(\opb{\hbar}u_1,\dots,\opb{\hbar}u_n)$. Identifying $\opb{\hbar}u_i$ with 
$\frac{\partial}{\partial{x_i}}$ we get an isomorphism $\DL\simeq \shd_\Lambda[[\hbar]]$.
\end{example}
From now on, and until the end of the proof of Proposition~\ref{pro:AL1} we work locally on $X$ and thus we may assume that there is an isomorphism $\gr\A\isoto\sho_X$.

We introduce a filtration $F\Ah$ on $\Ah$ by setting 
\eq\label{eq:newfilt}
&&F_k\Ah=\hbar^{-k}\A\mbox{ for }k\in\Z. 
\eneq
Therefore, there  is a natural isomorphism 
\eq
&&\grr_k^F\Ah\simeq T^{-k}\sho_X \mbox{ given by $\hbar\longleftrightarrow T$.}
\eneq
We endow $\AL$ with the induced filtration, that is, 
\eqn
\quad F_k\AL=\AL\cap F_k\Ah.
\eneqn
Recall (see~\cite{KS12}*{\S~1.4}) that for  a left Noetherian $\Chbar$-algebra $\shr$,
one says that a coherent $\shr$-module $\shp$  is locally projective
if the functor
\eqn
&&\hom[\shr](\shp,\scbul)\cl \mdcoh[\shr]\to\md[\Chbar_X]
\eneqn
is exact.
This is equivalent to  each  of the following conditions: (i) for each $x\in X$, the stalk
$\shp_x$ is projective as an $\shr_x$-module, (ii) for each $x\in X$, the stalk  $\shp_x$ is a flat $\shr_x$-module, 
(iii) $\shp$ is locally  a direct summand of a free $\shr$-module of finite rank.

Recall that one says that a ring $R$ has global homological  dimension $\leq d$ if both $\md[R]$ and $\md[R^\op]$ have 
homological dimension $\leq d$ (see~\cite{KS90}*{Exe.~I.28}). In such a case, we shall  write for short $\ghd(R)\leq d$. 

Also recall that $d_X$ denotes the complex dimension of $X$.
\begin{lemma}\label{le:grAL1} One has
\banum
\item $(\AL)^\loc\simeq\Ah$.
\item
The algebra  $\grr^F\AL$ is a graded commutative subalgebra of $\grr^F\Ah$.
\item
There are natural isomorphisms
$$\grr^F\AL\simeq \bigoplus_{k\in\Z}T^{-k}I_\Lambda^k\quad\text{and}\quad
\grr_{\geq0}^F\AL\simeq \bigoplus_{k\geq0}T^{-k}I_\Lambda^k,$$
where $I_\Lambda^k\seteq\sho_X$ for $k\le0$. 
\item
The sheaves of algebras $\grr^F\AL$ and $\grr^F_{\geq0}\AL$ are Noetherian.
\item
 For any $x\in X$, one has  $\ghd(\grr^F\AL)_x\leq d_X+1$. 
\ee
\end{lemma}

\begin{proof}
(a) is obvious since $\A\subset\AL\subset\Ah$. 

\spa
(b)  is obvious.

\spa
(c) $\grr_1^F(\AL)\simeq I_\Lambda$. Hence, $\grr_k^F\AL\simeq I_\Lambda^k$.

\spa
(d)  The commutative algebras  $\grr^F\AL$ and  $\grr_{\ge0}^F\AL$ 
are locally finitely presented  $\sho_X$-algebras.
Hence they are Noetherian.
(Note that the associated variety with  $\grr^F\AL$ is the deformation of normal bundle to $\La$.)

\spa
(e) For $x\in X$, set $R_x=(\grr^F\AL)_x$.
If $x\notin\Lambda$ , then $R_x\simeq \sho_{X,x}[T,\opb{T}]$  and $\ghd(R_x)\leq d_X+1$.
Assume now that $x\in\Lambda$. Then 
$R_x/TR_x\simeq \sho_{\Lambda,x}[y_1,\dots,y_n]$ (with $n=\codim_X\Lambda$)  has global homological  dimension $d_X$ and 
 $R_x[\opb{T}]\simeq \sho_{X,x}[T,\opb{T}]$ has global homological  dimension $d_X+1$. 
 Hence, $\ghd(R_x)\leq d_X+1$   by the classical  Lemma~\ref{le:alg} below.
\end{proof}

\begin{lemma}\label{le:alg}
Let $R$ be a commutative Noetherian ring and let $t\in R$ be a non-zero divisor. Assume that $R/tR$
 has global homological dimension $\leq d$ 
 and the localization $R[\opb{t}]$ 
has global homological dimension $\leq d+1$. 
Then $R$ has global homological dimension $\leq d+1$. 
\end{lemma}
\begin{proof}
(i) Let $\Spec(R)$ denote as usual the set of prime ideals  of $R$.
For $\prim\in \Spec(R)$,  denote by $R_\prim$ the localization of $R$ at $\prim$. It is well-known that $R$ has global homological  dimension $\leq d$ if and only if
for any $\prim\in\Spec(R)$, $R_\prim$  has global homological  dimension $\leq d$. 

\spa
(ii) Let $\prim\in \Spec(R)$ and assume that $t\notin \prim$. Then $R_\prim\simeq (R[\opb{t}])_\prim$ has
global homological  dimension   $\leq d+1$.

\spa
(iii) Let $\prim\in \Spec(R)$ and assume that $t\in \prim$. In this case, $R_\prim/tR_\prim\simeq (R/tR)_\prim$ has
global homological  dimension   $\leq d$.
This implies that $R_\prim$ is a regular  local ring of global homological  dimension   $\leq d+1$.
\end{proof}

\begin{proposition}[{see~\cite{KS12}*{Lem.~7.1.3} in the symplectic case}]\label{pro:AL1}
One has
\banum
\item
the  $\Chbar$-algebroid $\AL$ is right and left Noetherian,
\item
$\gr(\shn)\in \Derb_\coh(\gr\AL)$
for any $\shn\in\Derb_\coh(\AL)$.
\eanum
\end{proposition}
\begin{proof}
(a)  follows from Proposition~\ref{pro:ka03} since  $\sha$ is Noetherian by Theorem~\ref{th:ks125},
$\grr_{\geq0}\AL$ is Noetherian  by Lemma~\ref{le:grAL1} and the $I^k_\Lambda$'s are coherent $\A$-modules since they are coherent 
$\sho_X$-modules.

\spa
(b)  Let us represent $\shn$ by a complex $\shl^{\scbul}$ bounded from above of locally free $\AL$-modules of finite rank. Then $H^i(\shl^{\scbul})\simeq0$ for $i\ll0$. Replacing $\shl^{\scbul}$ with $\tau^{\geq j}\shl^{\scbul}$ for $j\ll 0$ we find a bounded complex $\shl^\scbul$ of coherent $\AL$-modules for which $\hbar$ is injective. Now 
$\gr(\shn)$ is represented by the complex $\shl^\scbul/\hbar \shl^\scbul$
and the result follows.

\spa
(c) Let $d$ denote the projective dimension 
\end{proof}

In the sequel, for $\shn\in\Derb(\AL)$ we set
\eq\label{eq:notgrl}
&&\grl(\shn)\eqdot  \gr(\DL\ltens[\AL]\shn)\simeq    \D[\Lambda]\ltens[\gr(\AL)]\gr(\shn).
\eneq

\begin{corollary}\label{cor:ayl0}
If $\shn\in\Derb_\coh(\AL)$, then $\grl(\shn)\in\Derb_\coh(\D[\Lambda])$ and $\chv(\grl(\shn))$ is a closed $\C^\times$-conic complex analytic subset of $T^*\Lambda$.
\end{corollary}
\begin{proof}
By Proposition~\ref{pro:AL1}~(b) and Lemma~\ref{le:grAL1}~(e), $\gr\shn$ is locally quasi-isomorphic to a bounded complex of projective  $\gr\AL$-modules of finite type. To conclude, note that if $\shp$ is a projective  $\gr\AL$-modules of finite type, then $\D[\Lambda]\ltens[\gr(\AL)]\gr(\shp)$ is concentrated in degree $0$ and is $\D[\Lambda]$-coherent. The result for $\chv(\grl(\shn))$ follows.
\end{proof}

\begin{proposition}[{see~\cite{KS12}*{Prop.~7.1.8}} in the symplectic case]\label{pro:SScar2b}
Let~$\shn$ be a coherent $\AL$-module.  Then
\eq
&&\rhom[\AL](\shn,\shl)\in\Derb(\Chbar_X),\label{eq:inDb}\\
&&\SSi(\rhom[\AL](\shn,\shl))=\chv(\grl\shn)\label{eq:chvn}.
\eneq
\end{proposition}
\begin{proof}
(i) One has
\eqn
&&\rhom[\AL](\shn,\shl)\simeq\rhom[\DL](\DL\ltens[\AL]\shn,\shl).
\eneqn
Set $F=\rhom[\DL](\DL\ltens[\AL]\shn,\shl)$. Then $F\in\RD^+(\Chbar_X)$, 
 $F$ is cohomologically
complete by Proposition~\ref{prop:homcc} and $\gr(F)\simeq \rhom[{\D[\Lambda]}](\grl\shn,\sho_\Lambda)$. 

\spa
(ii) We have $\gr F\in\Derb(\Chbar_X)$ by Lemma~\ref{le:grAL1}~(c). This implies~\eqref{eq:inDb} by Proposition~\ref{pro:158}.

\spa
(iii) We have  $\SSi(F)=\SSi(\gr(F))$ by Proposition~\ref{pro:sscohco}. On the other hand,
$\gr(F)\simeq \rhom[{\D[\Lambda]}](\grl\shn,\sho_\Lambda)$ and the microsupport of this complex is equal
to~$\chv(\grl\shn)$ by~\cite{KS90}*{Th~11.3.3}.
\end{proof}

\begin{definition}\label{def:ALlattice}
A coherent $\AL$-submodule $\shn$ of a coherent $\Ah[X]$-module $\shm$ is called
an $\AL$-lattice
of~$\shm$ if $\shn$ generates $\shm$ as an $\Ah[X]$-module.
\end{definition}
One easily proves that if $\shn$ is an $\AL$-lattice of $\shm$, then $\chv(\gr\shn)$ depends only on $\shm$. 
\begin{notation}
For a coherent  $\Ah[X]$-module $\shm$, one sets 
 $\Chl(\shm)\eqdot \chv(\grl\shn)$ for $\shn$ a (locally defined)  $\AL$-lattice of $\shm$.
\end{notation}

\subsection{Reminders on holonomic \DQ modules}

We shall recall here the main results of~\cite{KS12}*{Ch.~7}. 

In this subsection, we assume that $X$ is symplectic and that $\Lambda$ is Lagrangian. 
In this case, $\gr(\AL)\simeq\D[\Lambda]$ as an algebroid
and thus $\grl(\shn)\simeq\gr(\shn)$.
\begin{definition}Assume that $X$ is symplectic and $\Lambda$ is Lagrangian. 
An object $\shn$ of~$\Derb_\coh(\AL)$ is holonomic if 
$\gr(\shn)$ belongs to~$\Db_\hol(\shd_\Lambda)$.
\end{definition}

\begin{theorem}[{see~\cite{KS12}*{Th.~7.1.10}}]\label{th:7110}
Assume that $X$ is symplectic. 
Let~$\shn$ be a holonomic $\AL$-module.
\banum
\item
The objects $\rhom[\AL](\shn,\shl)$ and $\rhom[\AL](\shl,\shn)$
belong to~$\Derb_\Cc(\Chbar_\Lambda)$ and their microsupports
are contained in~$\chv(\gr\shn)$.
\item
There is a natural isomorphism in~$\Derb_\Cc(\Chbar_\Lambda)$
\eq\label{eq:dualDhh}
&&\rhom[\AL](\shn,\shl)\isoto\RDD_X\bl\rhom[\AL](\shl,\shn)\br\,[d_X].
\eneq
\eanum
\end{theorem}

The crucial result in order to prove Theorem~\ref{th:723}  below is the following.

\begin{proposition}[{see~\cite{KS12}*{Prop.~7.1.16}}]\label{pro:7116}
Assume that $X$ is symplectic and $\Lambda$ is Lagrangian. For a coherent $\Ah$-module $\shm$, we have
\eqn
&&\codim\Chl(\shm)\ge\codim\Supp(\shm).
\eneqn
\end{proposition}

The next result is a variation on a classical theorem of~\cite{Ka75} on holonomic D-modules.

\begin{theorem}[{see~\cite{KS12}*{Th.~7.2.3}}]\label{th:723}
Assume that $X$ is symplectic. 
Let~$\shm$ and $\shn$ be two holonomic $\Ah$-modules. Then 
\bnum
\item
the object $\rhom[{\Ah}](\shm,\shn)$ belongs to~$\Derb_\Cc(\Chbarl_X)$,
\item
there is a canonical isomorphism:
\eq\label{eq:dualmorph1}
&&\rhom[{\Ah}](\shm,\shn)
\isoto\bl\RDD_X\rhom[{\Ah}](\shn,\shm)\br\,[d_X],
\eneq
\item
the object $\rhom[{\Ah}](\shm,\shn)[d_X/2]$ is perverse.
\enum
\end{theorem}

\section{Proof of the main theorems and an example}

\subsection{Proof of Theorem~\ref{th:Main1}}
In this subsection,  $X$ is again a complex Poisson manifold endowed with a \DQ-algebroid $\A$.

By using the diagonal procedure, we may assume that $\shl=\shl_0^\loc$ with $\shl_0$ an $\A[X]$-module simple along $\Lambda$. 
By the hypothesis, we may find an $\AL$-lattice $\shn$ of $\shm$. 
 Set 
\eq\label{eq:F}
&&F_0\eqdot \rhom[\AL](\shn,\shl_0),\quad F\eqdot \rhom[{\Ah[X]}](\shm,\shl)\simeq F^\loc.
\eneq
One knows by Theorem~\ref{th:723} that $F\vert_Y\in\Derb_\Cc(\Chbarl_{Y\cap\Lambda})$ and 
 one knows by Proposition~\ref{pro:SScar2b} and Corollary~\ref{cor:ayl0} that 
$\SSi(F_0)\times_{\Lambda}(\Lambda\cap Y)$ is Lagrangian and subanalytic in $T^*\Lambda$. 
Since $\SSi(F)\subset\SSi(F_0)$ by Proposition~\ref{pro:ssloc}, it remains to apply Corollary~\ref{cor:finite1}.

\subsection{Proof of Theorem~\ref{th:Main2}}

In this subsection, $X$ is a quasi-compact separated smooth algebraic variety over $\C$ endowed with the Zariski topology. For an algebraic variety $X$, one denotes by  $\Xan$ the  complex analytic manifold associated with $X$ and by $\rho\cl \Xan\to X$ the natural map. There is a natural morphism $\opb{\rho}\sho_X\to\sho_\Xan$ and it is well-known that this morphism is faithfully flat (cf~\cite{Se56}).

\Lemma\label{lem:analy}
Let $\A$ be a \DQ-algebroid on $X$.
Then there exists a \DQ-algebroid $\A[\Xan]$  on $X_\an$ 
  together with a  functor  $\opb{\rho}\A\to\A[\Xan]$.
Moreover such an $\A[\Xan]$ is unique up to a unique isomorphism.
\enlemma
\Proof
First, consider a star algebra $\sha=(\Oh,\star)$ on a smooth
 algebraic variety $X$. The star product is defined by a sequence of algebraic bidifferential operators $\{P_i\}_i$ (see~\cite{KS12}*{Def.~2.2.2}) and one defines 
a star algebra $\sha^\an=(\Oh[\Xan],\star)$ on $X_\an$
by using the same bidifferential operators.

There exists an open (for the Zariski topology) covering 
$X=\bigcup_{i\in I}U_i$ 
such that, for each $i$, there exists an object $s_i$ of the category $\A(U_i)$.
Then $\sha_i\seteq\shend(s_i)$ is a star algebra.
For $i,j\in I$, since $s_i\vert_{U_i\cap U_j}$ and $s_j\vert_{U_i\cap U_j}$
are locally isomorphic, there exists an open covering  $U_i\cap U_j=\bigcup_{a\in A_{ij}}U_{ij}^a$ such that setting 
 $U_{ij}=\bigsqcup_{a\in A_{ij}}U_{ij}^a$, there exist
 an isomorphism $\alpha_{ij}\cl s_i\vert_{U_{ij}}\isoto s_j\vert_{U_{ij}}$.
Then we have 
$$a_{ijk}\seteq\alpha_{ij}\alpha_{jk}\alpha_{ki}\in
\End(s_i\vert_{U_{ijk}})=\sha_i(U_{ijk}),$$
where $U_{ijk}=U_{ij}\times_XU_{jk}\times_XU_{ki}$.

Hence we have an isomorphism $\beta_{ij}\cl\sha_i\vert_{U_{ij}}
\isoto\sha_j\vert_{U_{ij}}$
defined by $\sha_i\ni a\mapsto \alpha_{ij}\circ a\circ \alpha_{ij}^{-1}\in\sha_j$.
Moreover they satisfy the compatibility condition:
$$\beta_{ij}\beta_{jk}\beta_{ki}=\Ad(a_{ijk})\in\End(\sha_i\vert_{U_{ijk}}).$$
Then the data
$(\st{U_i},\st{U_{ij}},\st{\sha_i},\st{\beta_{i.j}},\st{a_{ijk}})$
satisfies the compatibility condition.
Conversely, we can recover $\A$  from such data  (see \cite{KS12}).

On $(U_{i})_\an$ we can define $\sha_{i}^\an$.
Similarly we can extend  $\beta_{ij}$
to $\beta^{\an}_{ij}\cl
\sha_{i}^{\an}\vert_{(U_{ij})_\an}
\isoto\sha_{j}^{\an}\vert_{(U_{ij})_\an}$
Finally we have
$a_{ijk}\in\sha_i(U_{ijk})\subset\sha_{i}^\an\bl(U_{ijk})_\an\br$,
Then the data
$$(\st{(U_{i})_\an},\st{(U_{ij})_\an},\st{\sha_{i}^\an},\st{\beta^{\an}_{i.j}},\st{a_{ijk}})$$
satisfies the compatibility condition, and it defines
a DQ-algebroid $\A[X_\an]$ on $X_\an$.
\QED

\begin{proposition}\label{pro:flat}
The algebroid $\A[\Xan]$ is faithfully flat over $\opb{\rho}\A$.
\end{proposition} 
\begin{proof}
It is enough to prove that for each $x\in X$,  $\sha_{\Xan,x}$ is faithfully flat over $\sha_{X,x}$. 
This follows from~\cite{KS12}*{Cor.~1.6.7} since 
$\sha_{X,x}/\hbar\sha_{X,x}\simeq\sho_{X,x}$ is Noetherian,
$\sha_{\Xan,x}$ is cohomologically complete and finally  
 $\sha_{\Xan,x}/\hbar\sha_{\Xan,x}\simeq\sho_{\Xan,x}$ is faithfully flat over $\sho_{X,x}$. 
\end{proof}

For an $\A$-module $\shm$ we set
\eqn
&&\shm_\an\eqdot \A[\Xan]\tens[\opb{\rho}\A]\opb{\rho}\shm.
\eneqn

\begin{proof}[Proof of Theorem~\ref{th:Main2}]
As in the proof of Theorem~\ref{th:Main1}, we may assume that $\shl\simeq\shl_0^\loc$ where $\shl_0$ is a simple $\A$-module along a smooth algebraic Lagrangian manifold $\Lambda$, the module $\shm$  remaining algebraically good. Choose an $\AL$-lattice $\shn$ of $\shm$. Let 
\eq\label{eq:Fan}
&&F_\an\eqdot\rhom[{\Ah[\Xan]}](\shm_\an,\shl_\an)\simeq \rhom[\ALan](\shn_\an,(\shl_0)_\an)^\loc. 
\eneq
By Proposition~\ref{pro:SScar2b} we know that 
$\SSi(F_\an)\subset\chv(\grl\shn_\an)$ and this set is contained in $\chv(\grl\shn)$ which is an algebraic Lagrangian subvariety of $T^*\Lambda$. To conclude, apply Corollary~\ref{cor:finite2}. 
\end{proof}

\begin{remark}\label{rem:2simple}
(i) If one assumes that $\shm$ and $\shl$ are simple modules along two smooth algebraic varieties $\Lambda_1$ and $\Lambda_2$ of $X$, which is the situation appearing  in~\cite{GJS19}, there is a much simpler  proof. Indeed, it follows from~\cite{KS12}*{Th.~7.4.3} that in this case
\eq\label{eq:SSFW}
&&\SSi(F)\subset C(\Lambda_1,\Lambda_2),
\eneq
the Whitney normal cone of $\Lambda_1$ along $\Lambda_2$ and this set is algebraic. Hence, it remains to apply Corollary~\ref{cor:finite2}. Note that Th.~7.4.3 of loc.\ cit.\ is a variation on~\cite{KS08}. 

\vspace{3ex}\noindent
(ii) Also remark that~\eqref{eq:SSFW} is no more true in the general case of irregular holonomic modules and until now, there is no estimate of $\SSi(F)$, except of course,  the fact that it is a Lagrangian set.
\end{remark}

\subsection{An example}
Consider the Poisson manifold $X=\C^4$ with coordinates $(x_1,x_2,y_1,y_2)$, the Poisson bracket being defined by:
\eq\label{eq:poisson1}
&&\ba{c}
\{x_1,x_2\}=0,\, \{y_1,x_1\}=\{y_2,x_2\}=x_1, \\
\{y_1,y_2\}=y_2, \{y_1,x_2\}=y_2, \{y_1,x_2\}=\{y_2,x_1\}=0.
\ea\eneq
Denote by $\A$ the \DQ-algebra defined by the relations $y_1=\hbar x_1\partial_{x_1}$, $y_2=\hbar x_1\partial_{x_2}$, that is,
\eq\label{eq:poisson2}
&&\ba{c}
[x_1,x_2]=0,\, [y_1,x_1]=[y_2,x_2]=\hbar x_1, [y_1,y_2]=\hbar y_2, \\
{} [y_1,x_2]=\hbar y_2, [y_1,x_2]=[y_2,x_1]=0.
\ea
\eneq
Hence, $Y=\{x_1\neq0\}$ is the symplectic locus  $X\setminus \Xns$ 
of the Poisson manifold $X$. Set $\Lambda=\{y_1=y_2=0\}$. Then $\Lambda\cap Y$ is Lagrangian in $Y$. 

Define the $\A$-module $\shl$ by $\shl=\A\cdot u$ with the relations $y_1u=y_2u=0$. Then $\shl\simeq\sho_\Lambda^\hbar$ and for $a(x)\in\sho_\Lambda^\hbar$, one has
\eqn
&&\left\{\parbox{75ex}{
$y_1a(x)u=\hbar x_1\frac{\partial a}{\partial x_1} u$\\
$y_2a(x)u=\hbar x_1\frac{\partial a}{\partial x_2} u.$
}\right.
\eneqn
Now define the left $\A$ module $\shm$ by  $\shm=\A\cdot v$ with the relations $(y_1+\hbar)v =y_2v=0$. 
Then  the complex below, in which the operators act on the right
\eq\label{eq:resolution1}
&&\xymatrix{
0&\shm\ar[l]&\A\ar[l]&&\A^{\oplus 2}\ar[ll]^-
{\small\scbuld\left( \begin{array}{cc}y_1+\hbar\\ y_2\end{array}\right)}
&&\A\ar[ll]^-{\scbuld(y_2,-y_1)}&0\ar[l] 
}\eneq
is a free resolution of $\shm$. 

Hence, the object  $\rhom[\A](\shm,\shl^\loc)$ is represented by the complex  (the operators act on the left)
\eq\label{eq:resolution2}
&&\xymatrix{
0\ar[r]&\Ohl[\Lambda]\ar[rr]_-
{\small\left( \begin{array}{cc}x_1\partial_{x_1} +1\\ x_1\partial_{x_2}\end{array}\right)\scbuld}
&&(\Ohl[\Lambda])^{\oplus 2}\ar[rr]_-{(x_1\partial_{x_2}, -x_1\partial_{x_1})\scbuld}&&\Ohl[\Lambda]\ar[r]&0.
}\eneq

Since $x_1\partial_{x_1}\Ohl[\Lambda]+x_1\partial_{x_2}\Ohl[\Lambda]
=x_1\Ohl[\Lambda]$ and
$\Ohl[\Lambda]/x_1\Ohl[\Lambda]\simeq \Ohl[\Lambda\cap\{x_1=0\}]$, we have

\eqn
&&\ext[\A]{2}(\shm,\shl^\loc)\simeq\Ohl[\Lambda\cap\{x_1=0\}].
\eneqn
This example shows that $\rhom[\A](\shm,\shl^\loc)$ does not belong to $\Derb_\Cc(\Chbarl)$.

\vspace*{1cm}
\noindent
\parbox[t]{21em}
{\scriptsize{
\medskip\noindent
Masaki Kashiwara\\
Kyoto University Institute for Advanced Study,\\
Research Institute for Mathematical Sciences\\
Kyoto University, 606--8502, Japan\\
and\\
Department of Mathematical Sciences and School of Mathematics,\\
Korean Institute for Advanced Studies, \\
Seoul 130-722, Korea

\medskip\noindent
Pierre Schapira\\
Sorbonne University, CNRS IMJ-PRG\\
4, place Jussieu F-75005 Paris France \\
e-mail: pierre.schapira@imj-prg.fr\\
http://webusers.imj-prg.fr/\textasciitilde pierre.schapira/
}}

\end{document}